\documentclass[12pt,reqno]{amsart}
\usepackage{amsmath}
\usepackage{amssymb}
\usepackage{amstext}
\usepackage{mathrsfs}
\usepackage{a4wide}
\usepackage{graphicx}
\usepackage{bm}
\allowdisplaybreaks \numberwithin{equation}{section}
\usepackage{color}
\usepackage{cases}

\usepackage{hyperref}
\hypersetup{hypertex=true,
            colorlinks=true,
            linkcolor=blue,
            anchorcolor=blue,
            citecolor=blue }

\numberwithin{equation}{section}

\newtheorem{theorem}{Theorem}[section]
\newtheorem{proposition}[theorem]{Proposition}
\newtheorem{corollary}[theorem]{Corollary}
\newtheorem{lemma}[theorem]{Lemma}

\theoremstyle{definition}

\newtheorem{definition}[theorem]{Definition}

\theoremstyle{remark}
\newtheorem{remark}[theorem]{Remark}

\def\d{\mathrm{d}}

\newcommand{\R}{\mathbb{R}}

\newcommand{\D}{\mathbb{D}}

\newcommand{\dis}{\displaystyle}

\newcommand{\LL}{\mathcal{L}}

\newcommand{\M}{\mathcal{M}}
\newcommand{\T}{\mathcal{T}}
\newcommand{\s}{\mathcal{S}(\R^2)}

\begin{document}

\title[Radial symmetry of stationary and uniformly-rotating solutions]{Radial symmetry of stationary and uniformly-rotating solutions to the 2D Euler equation in a disc}

\author{Boquan Fan, Yuchen Wang, Weicheng Zhan}

\address{Institute of Applied Mathematics, Chinese Academy of Sciences, Beijing 100190, and University of Chinese Academy of Sciences, Beijing 100049,  P.R. China}
\email{fanboquan22@mails.ucas.ac.cn}

\address{Institute of Mathematics, University of Augsburg, Augsburg, 86159, Germany, School of Mathematics Science, Tianjin Normal University, Tianjin,  300387, P.R. China}
\email{yuchen.wang@uni-a.de}

\address{School of Mathematical Sciences, Xiamen University, Xiamen, Fujian, 361005, P.R. China}
\email{zhanweicheng@amss.ac.cn}

\begin{abstract}
We study the radial symmetry properties of stationary and uniformly rotating solutions of the 2D Euler equation in the unit disc, both in the smooth setting and the patch setting. In the patch setting, we prove that every uniformly rotating patch with angular velocity $\Omega\le 0$ or $\Omega \ge 1/2$ must be radial, where both bounds are sharp. The conclusion holds under the assumption that the rotating patch considered is disconnected, with its boundaries consisting of several Jordan curves. We also show that every uniformly rotating smooth solution $\omega_0$ must be radially symmetric if its angular velocity $\Omega\le \inf \omega_0/2$ or $\Omega\ge \sup \omega_0/2$. The proof is based on the symmetry properties of non-negative solutions to elliptic problems. A newly tailored approach is developed to address the symmetries of non-negative solutions to piecewise coupled semi-linear elliptic equations.
\end{abstract}

\maketitle

\tableofcontents

\bibliographystyle{plain}

\section{Introduction and Main results}

\subsection{The 2D Euler equation}
In this paper, we are concerned with radial symmetry properties for stationary/uniformly-rotating solutions to the two-dimensional incompressible Euler equation in the unit disc $\mathbb{D}$ of the Euclidean space $\R^2$. That system reads as follows:
\begin{align}\label{1-1}
\begin{cases}
\partial_t\omega+\mathbf{v}\cdot \nabla\omega=0,\ (x,t)\in \mathbb{D}\times\mathbb{R}_+,&\\
\mathbf{v}=\nabla^\perp (-\Delta)^{-1}\omega, &\\\
\omega|_{t=0}=\omega_0,
\end{cases}
\end{align}
where $\mathbf{v}=(v_1,v_2)$ is the velocity field and its vorticity is given by the scalar $\omega=\partial_1v_2-\partial_2v_1$, and $\nabla^\perp=(\partial_2, -\partial_1)$. The known function $\omega_0$ is the initial condition.

The first equation in \eqref{1-1} simply
means that the vorticity is constant along particle trajectories. The second equation in \eqref{1-1} tells us how to recover velocity from vorticity, which is referred to as the Biot-Savart law. Let $G(x, y)$ be the Green function for $-\Delta$ in $\mathbb{D}$ with zero Dirichlet data, namely,
\begin{equation*}
  G(x, y)=\frac{1}{2\pi}\ln\frac{1}{|x-y|}-h(x, y),\ \ \ h(x, y)=-\frac{1}{2\pi}\ln \left|\frac{x}{|x|}-|x|y \right|.
\end{equation*}
Then the Biot-Savart law can be written in the form
\begin{equation}\label{1-2}
  \mathbf{v}=\nabla^\perp \mathcal{G}[\omega],\ \ \ \mathcal{G}[\omega] (x):=\int_{\mathbb{D}}G(x, y)\omega(y)\d y.
\end{equation}
A convenient reference for these results is \cite[Chapter 2]{Maj2002MR1867882}. It is well-known since the work of Yudovich \cite{Yud1963MR158189} that any bounded and integrable initial datum $\omega_0$ generates a unique
global in time weak solution of \eqref{1-1}; see also \cite{Bur2005MR2186035}.

\subsection{Stationary/uniformly rotating solutions}

We will focus here on establishing radial symmetry properties for stationary or uniformly rotating solutions to Equation \eqref{1-1}. Our analysis will encompass both the patch setting and the smooth setting.

 Thanks to Yudovich's theorem, one can deal rigorously with vortex patches, which are vortices uniformly distributed in a bounded region $D$, i.e., $\omega_0=1_D$. Since the vorticity is transported along trajectories, we conclude that $\omega(x, t)=1_{D_t}$ for some region $D_t$. For a bounded open set $D\subset \mathbb{D}$, we say that an indicator function $\omega=1_D$ is a \emph{stationary patch solution} to \eqref{1-1} if $\mathbf{v}(x)\cdot \vec{n}(x)=0$ on $\partial D$, where $\mathbf{v}$ is given by \eqref{1-2}, and $\vec{n}(x)$ denotes the outward unit normal to $\partial D$ at $x$. This leads to the integral equation
\begin{equation}\label{1-3}
  \mathcal{G}[1_D]\equiv C_i\ \ \ \text{on}\ \partial D,
\end{equation}
where the constant $C_i$ can differ on different connected components of $\partial D$. Let $\Omega \in \mathbb{R}$ be a fixed real number. We denote by $e^{i\Omega t}$ the counter-clockwise rotation by an angle $\Omega t$ about the origin. We say that $\omega(x, t)=1_D(e^{-i\Omega t} x)$ is a \emph{uniformly rotating vortex patch} with angular velocity $\Omega$ if $1_D$ becomes stationary in a frame rotating with angular velocity $\Omega$, that is, $\left(\nabla^\perp \mathcal{G}[\omega]+\Omega x^\perp\right)\cdot \vec{n}(x)=0$ on $\partial D$. This is equivalent to
\begin{equation}\label{1-4}
  \mathcal{G}[1_D]+\frac{\Omega}{2}|x|^2\equiv C_i\ \ \ \text{on}\ \partial D,
\end{equation}
where the constant $C_i$ again can take different values along different connected components of $\partial D$. Note that a stationary patch $D$ also satisfies \eqref{1-4} with $\Omega=0$ and can be viewed as a special case of a uniformly rotating patch with zero angular velocity.

Similarly, we refer to $\omega(x, t) = \omega_0(e^{-i\Omega t} x)$ as a \emph{uniformly rotating smooth solution} of \eqref{1-1} with angular velocity $\Omega$ (reducing to a stationary solution when $\Omega = 0$), provided the vorticity profile $\omega_0 \in C^2(\overline{\D})$ satisfies $\left(\nabla^\perp \mathcal{G}[\omega_0] + \Omega x^\perp\right) \cdot \nabla \omega_0 = 0$. Consequently, the profile $\omega_0$ fulfills the equation
\begin{equation}\label{1-5}
   \mathcal{G}[\omega_0]+\frac{\Omega}{2}|x|^2\equiv C_i\ \ \ \text{on each connected component of a regular level set of}\ \omega_0,
\end{equation}
where $C_i$ may vary across different connected components of a given regular level set $\{\omega_0=c\}$.

\subsection{Motivations}
Let $B_r(x)$ denote the open disc in $\mathbb{R}^2$ with center $x$ and radius $r>0$. It is easy to see that
\begin{equation*}
\mathcal{G}[1_{B_r(0)}](x)=\begin{cases}\dis \frac{1}{4}\left(r^2+2r^2\ln\frac{1}{r}-|x|^2 \right),& 0\le |x|< r,\\
\dis \frac{r^2}{2}\ln \frac{1}{|x|},&r\le |x|\le 1.\end{cases}
\end{equation*}
Hence, every radially symmetric patch $1_{B_r(0)}$ with $0<r\le 1$ automatically satisfies \eqref{1-4} for all $\Omega\in \R$.

For each $\Omega\in (0, 1/2)$, de la Hoz-Hassainia-Hmidi-Mateu \cite{del2016MR3570233} showed that there exist non-radial uniformly rotating vortex patches with $m$-fold symmetry bifurcating from discs or annuli at $\Omega$. Note that a set is said to be $m$-fold symmetric if it is invariant under rotation by an angle of $2\pi/m$ around its center. They also presented some numerical experiments illustrating the interaction between the boundary of the patch and the rigid boundary.

A natural question arises: \emph{whether there exists a non-radial uniformly rotating vortex patch with angular velocity $\Omega\le 0$ or $\Omega\ge 1/2$.}

The first purpose of this paper is to give a complete answer to this question. We will show that every uniformly rotating vortex patch with angular velocity $\Omega\le 0$ or $\Omega\ge 1/2$ must be radial. Thus, the answer to the above question is no.

It can also be seen that every radially symmetric smooth function automatically satisfies \eqref{1-5} for all $\Omega\in \R$. The second objective of this paper is to establish analogous symmetry properties for uniformly rotating smooth solutions of \eqref{1-1}.  More precisely, we will show that every uniformly rotating smooth solution must be radially symmetric if its angular velocity $\Omega\le \inf \omega_0/2$ or $\Omega\ge \sup \omega_0/2$.

\subsection{Main results}

Our first main result is as follows:

\begin{theorem}\label{thm}
  Let $D\subset \mathbb{D}$ be an open set whose boundary consists of a finite collection of mutually disjoint Jordan curves. Assume $D$ is a stationary/uniformly rotating vortex patch of \eqref{1-1}, in the sense that $D$ satisfies \eqref{1-4} for some $\Omega\in \R$. Then $D$ must be radially symmetric if $\Omega\in (-\infty, 0]\cup [\frac{1}{2}, \infty)$.
\end{theorem}

Note that in Theorem \ref{thm}, $D$ is allowed to be disconnected, and each connected component may not be simply connected. If $D$ is assumed to be connected (i.e., $D$ is a domain), then the only possible configurations of $D$ are a disc or an annulus. Therefore, we have the following corollary.

\begin{corollary}\label{cor}
  Let $D\subset \mathbb{D}$ be a domain whose boundary consists of a finite collection of mutually disjoint Jordan curves. Assume $D$ is a stationary/uniformly rotating vortex patch of \eqref{1-1}, in the sense that $D$ satisfies \eqref{1-4} for some $\Omega\in (-\infty, 0]\cup [\frac{1}{2}, \infty)$. Then $D$ must be either a disc or an annulus. More precisely, we have that
  \begin{itemize}
    \item [(1)]$D$ must be a disc centered at the origin if it is simply connected;
    \item [(2)]$D$ must be an annulus centered at the origin if it is not simply connected.
  \end{itemize}
\end{corollary}

\begin{remark}
As mentioned above, the existence of non-radial uniformly rotating vortex patches with any angular velocity $\Omega\in (0, 1/2)$ was established by de la Hoz, Hassainia, Hmidi, and Mateu in \cite{del2016MR3570233}. In this sense, our rigidity result in Theorem \ref{thm} is sharp.
\end{remark}

\begin{remark}
Similar rigidity results for the whole-plane case was established by G\'{o}mez-Serrano, Park, Shi, and Yao in \cite{Gom2021MR4312192}. The boundary regularity condition for the patches therein appears to require, at a minimum, the assumption \textbf{(HD)}; see subsection 4.2 in \cite{Gom2021MR4312192}. As remarked therein, a rectifiable boundary satisfies this assumption. However, as can be seen, our boundary regularity condition for the patches is very mild. Note that a Jordan curve may have positive two-dimensional Lebesgue measure; see \cite[Chapter 8]{Sagan1994MR1299533}. It is interesting to consider whether the boundary regularity condition in \cite{Gom2021MR4312192} can be extended to Jordan curves. We believe this is the case.
\end{remark}

Our second main result concerns the radial symmetry of uniformly rotating smooth solutions.

\begin{theorem}\label{thm2}
  Let $\omega_0\in C^2(\overline{\D})$. Assume $\omega(x, t)=\omega_0(e^{-i\Omega t} x)$ is a stationary/uniformly rotating smooth solution  of \eqref{1-1}, in the sense that it satisfies \eqref{1-5} for some $\Omega\in \R$. Then $\omega_0$ must be radially symmetric if $\Omega\le \inf \omega_0/2$ or $\Omega\ge \sup \omega_0/2$.
\end{theorem}

\begin{remark}
By a suitable modification of the proof of Theorem \ref{thm}, we can also generalize the result presented in Theorem \ref{thm} to encompass stationary and uniformly rotating vortex patches $\omega$ that exhibit multiple scales, with the allowance for a change in sign within these patches; see Remark \ref{re5} in Section \ref{s5}. The new thresholds of angular velocity are also given by $\inf \omega/2$ and $\sup \omega/2$, respectively. We believe that these bounds are sharp, although we have not yet found a complete result on the existence of non-radial solutions in the literature.
\end{remark}

\subsection{Comments on related works}
There is extensive literature on the numerical and analytical structures of stationary/uniformly rotating solutions. Below, we briefly review some related works in both the entire plane and disc cases. We will primarily focus on the rigidity (only trivial solutions exist), with some flexibility (non-trivial solutions exist) addressed in passing. Rigidity properties of relative equilibria (which do not change shape as time evolves) in an ideal fluid are an interesting topic in fluid mechanics. The core issue is to identify sufficient rigidity conditions under which the fluid motion must assume a simple geometric configuration.

\subsubsection*{\textbf{The patch setting}}
Let us begin by the patch setting. In the literature, uniformly rotating vortex patches are commonly referred to as $V$-states. The first example of explicit non-trivial $V$-states was discovered by Kirchhoff \cite{Kirchhoff}, which is an ellipse of semi-axes $a$ and $b$ is subject to a perpetual rotation with the angular velocity $ab/(a+b)^2$. In the 1970s, Deem and Zabusky \cite{Deem1978} gave some numerical observations of the existence of more $V$-states with $m$-fold symmetry for $m\ge 2$. A few years later, Burbea gave in \cite{Bur1982MR646163} gave an analytical proof of their existence based on a conformal mapping parametrization and local bifurcation theory. Burbea’s approach was revisited in \cite{Hmi2013MR3054601} with more details and explanations. Recently, Burbea's branches of solutions were extended to global ones \cite{Has2020MR4156612}. The existence of doubly connected $V$-states with $m$-fold symmetry bifurcating from annuli was established in \cite{del2016MR3507551, Hmi2016MR3545942, wang2024degeneratebifurcationstwofolddoublyconnected}. The existence of disconnected non-trivial $V$-states has also been established by various methods; see, e.g., \cite{Gar2021MR4284365, Has2022MR4510650, Hmi2017MR3607460, Tur85MR783583} and the references therein.

All of the results mentioned above are established in the context of the entire plane. The existence of non-trivial $V$-states in discs appears to be relatively few. Note that due to boundary effects, the ellipses are no longer rotating vortex patches in the disc. As mentioned above, de la Hoz-Hassainia-Hmidi-Mateu \cite{del2016MR3570233} proved the existence of $m$-fold $V$-states in the disc, both for the simply connected and doubly connected cases, using bifurcation theory. It follows that non-radial $V$-states can exist for every $\Omega\in (0, 1/2)$. Some new phenomena, distinct from those in the entire plane, were also exposed in \cite{del2016MR3570233}. For example, curves of $1$-fold $V$-states exist close to Rankine vortices, which are not necessarily centered at the origin. Furthermore, the multiplicity of solutions with identical angular velocity also arises in the disc case. These results are due to the interaction between the patch and the rigid boundary $\partial \mathbb{D}$. Recently, Cao-Wan-Wang-Zhan \cite{Cao2021MR4191331} also constructed a family of non-trivial $V$-states with small positive angular velocity, which may not be centered at the origin, using the variational method.

It should be pointed out that all of the aforementioned non-trivial $V$-states have an angular velocity in the interval $(0, 1/2)$.

Let us first briefly review some rigidity results concerning $V$-states in the entire plane case. For stationary patches ($\Omega=0$), Fraenkel \cite[Chapter 4]{Fra2000MR1751289} showed that if $D$ possesses the \emph{same} constant $C$ on the whole $\partial D$, then $D$ must be a disc; see also \cite{Lu2012MR2890967}. As a direct consequence, every simply connected stationary patch must be a disc. Hmidi \cite{Hmi2015MR3427065} proved that any simply connected $V$-state with angular velocity $\Omega<0$ or $\Omega=1/2$ must be a disc, under an additional convexity assumption. Recently, G\'{o}mez-Serrano, Park, Shi and Yao \cite{Gom2021MR4312192} established a complete rigidity result, which asserts that every $V$-state with angular velocity $\Omega\le0$ or $\Omega\ge1/2$ must be radial. In addition, they also study the radial symmetry properties of stationary and uniformly rotating smooth solutions.

For $V$-states in a disc, symmetry results appear to be extremely rare. It is known that every simply connected stationary patch must be a disc. Indeed, in this case, the stream function $\mathcal{G}[1_D]$ solves a semilinear elliptic equation with a monotone nonlinearity, and the desired symmetry follows by applying the moving plane method developed in \cite{Ser1971MR333220} (see also \cite{Fra2000MR1751289}). Recently, Wang and Zuo \cite{Wang2021MR4172850} established symmetry results for $V$-states under a self-referential condition on their angular velocity. More precisely, they showed that every simply connected $V$-state $D$ must be a disc if its angular velocity $\Omega\le -\frac{2\ell^2}{(1-\ell^2)^2}$ or $\Omega\ge \max\left\{\frac{1}{2}, \frac{2\ell^2}{(1-\ell^2)^2}\right\}$, where $\ell:=\sup_{x\in D}|x|$. Note that the threshold deteriorates significantly as the patch approaches the rigid boundary. The main idea of their proof is borrowed from \cite{Gom2021MR4312192}. However, due to boundary effects, a related estimate involving the regular part $h(x, y)$ of the Green's function arises, which appears to be difficult to handle; see Lemma 3.4 in \cite{Wang2021MR4172850} for attempts at estimation. Whether a $V$-state with angular velocity $\Omega\not\in(0, 1/2)$ in a disc must be radial still remained an open question. The main motivation of the present paper is to completely resolve this open question.

\subsubsection*{\textbf{The smooth setting}}

We now turn to the smooth setting. The literature in this area is extensive; we therefore highlight only a few relevant works. For a detailed bibliography on the rigidity and flexibility properties of stationary/uniformly rotating Euler flows
in two dimensions and their applications, we refer to the recent papers \cite{Gom2024, Gom2021MR4312192, Hamel2023} and the references therein.

In the entire plane, any radial smooth vorticity with compact support automatically corresponds to a stationary solution. For the radial vorticity $\omega$ with zero-average, $\int_{\R^2}\omega(x)\,\d x=0$, its velocity vanishes outside the support of $\omega$. The stationary solutions that are obtained by patching together such radial solutions with disjoint supports are referred to as locally radial. Recently, compactly supported stationary Euler flows in two dimensions that are not locally radial were constructed in \cite{Enci2024, Enci2023, Gom2024}. Note that, as an immediate consequence of the divergence theorem, the vorticity of a compactly supported velocity field must change sign.

Regarding uniformly rotating smooth solutions ($\Omega \neq 0$), Castro-C\'{o}rdoba-G\'{o}mez-Serrano \cite{Cas2019MR3900813} desingularized a vortex patch to construct an $m$-fold rotating solution with $\Omega \sim \frac{m-1}{2m} > 0$ for $m \geq 2$. Garc\'{\i}a-Hmidi-Soler \cite{Gar2020MR4134155} considered the construction of $V$-states bifurcating from other radial profiles, including Gaussians and piecewise quadratic functions.

Shvydkoy-Luo \cite{Luo2015MR3359160, Luo2017MR3620896} classified the set of stationary smooth solutions of the form $\mathbf{v}=\nabla^\perp (r^\gamma f(\theta))$, where $(r, \theta)$ are polar coordinates. Hamel-Nadirashvili \cite{Ham2019MR3951689} proved that in the whole plane, bounded steady flows with no stagnation points are parallel shear flows. G\'{o}mez-Serrano-Park-Shi-Yao \cite{Gom2021MR4312192} proved that any uniformly rotating smooth solutions with angular velocity $\Omega\le 0$ must be radial. Recently, Gui-Xie-Xu \cite{Gui2024} provided a classification of stationary Euler flows in terms of the set of flow angles.

All the rigidity results mentioned above are established in the context of the entire plane (or strip, or half-plane). We now turn to the disc case. Hamel-Nadirashvili \cite{Hamel2023MR4556785} proved that stationary Euler flows without stagnation points in punctured discs (with certain conditions at the center) must be circular flows. Recently, Wang-Zhan \cite{Wang2023} showed that steady Euler flows in a disc with exactly one interior stagnation point and tangential boundary conditions must be circular flows. Besides these two results, we are not aware of any other direct rigidity results for stationary or uniformly rotating smooth solutions of \eqref{1-1} in a disc. Our Theorem \ref{thm2} provides a rigidity result for stationary or uniformly rotating smooth Euler flows in a disc from the perspective of vorticity.

\bigskip

\bigskip

\subsection{Idea of the proof}
Our strategy is to focus on the corresponding stream function of the flow and the elliptic problem it satisfies. The problem reduces to proving the symmetry properties of the positive solutions to the associated elliptic problem. In the patch setting, the main difficulty is that it seems challenging to prove directly that the stream function satisfies a single semi-linear elliptic equation over the entire domain $\D$. Fortunately, the stream function of the rotating solution satisfies a single semi-linear elliptic equation, at least slice by slice. We use the continuous Steiner symmetrization method, developed by Brock \cite{Bro1995MR1330619, Broc2000MR1758811, Bro2016MR3509374}, to derive some local symmetry properties of the solution. The key observation is that the implementation of this method can be localized, thus overcoming the difficulty of non-holistically satisfying semi-linear elliptic equations. The condition on the angular velocity ensures that the stream function is either weakly superharmonic or subharmonic. The combination of local symmetry and the superharmonic/subharmonic property ensures global radial symmetry. This completes the proof of the patch case. After completing the proof for the patch case, we extend the same approach to the smooth case, thereby concluding the proof.

\subsection{Organization of the paper}
The rest of the paper is organized as follows: In Section \ref{s2}, we present some preliminary results that will be used in the sequel. Section \ref{s3} is devoted to proving the symmetry results for stationary vortex patches. The proof is divided into several cases, which are considered individually, with the specific details presented in full. Sections \ref{s4} and \ref{s5} address the radial symmetry of uniformly rotating vortex patches with angular velocity $\Omega<0$ and $\Omega\ge 1/2$, respectively. The proof for these cases is essentially similar to the stationary case and even simpler, so it is only presented in outline. In Section \ref{s6}, we further establish the radial symmetry of uniformly rotating smooth solutions.

\subsection{Notation}Throughout the paper, we use the following notations.
\begin{itemize}
  \item Let $B_r(x)$ and $Q_r(x)$ denote the open and closed discs in $\mathbb{R}^2$ with center $x$ and radius $r>0$, respectively. For convenience, we write $B_r:=B_r(0)$ and $Q_r:=Q_r(0)$.
  \item The non-negative part of a function $f$ is denoted by $f_+$, i.e., $f_+=\max\{f, 0\}$.
  \item For a set $D$, we use $1_D$ to denote its indicator function.
  \item For a Jordan curve $\Gamma$, let $\operatorname{int}(\Gamma)$ denote its interior, defined as the bounded connected component of $\mathbb{R}^2$ separated by $\Gamma$. By the Jordan curve theorem, $\operatorname{int}(\Gamma)$ is open and simply connected.
  \item We say that two disjoint simple closed curves $\Gamma_1$ and $\Gamma_2$ are nested if $\Gamma_1\subset \text{int}(\Gamma_2)$ or vice versa. We say that two domains $D_1$, $D_2$ are nested if one is contained
in a hole of the other one.
 \item For a bounded domain $D\subset \R^2$, we denote by $\partial_{\text{out}}D$ its outer boundary. And if $D$ is doubly connected, then we denote by $\partial_{\text{in}}D$ its inner boundary.
 \item We adhere to the convention of denoting by $C$ a constant independent of the relevant parameters under consideration. The constant may change its actual value at different occurrences.
\end{itemize}

\section{Preliminary results}\label{s2}

In this section, we collect several preliminary results which will be frequently used in the proofs of main results.

\subsection{The continuous Steiner symmetrization}
As mentioned above, the key tool in our proof is the continuous Steiner symmetrization method. For the convenience of the reader we first recall the definition of the continuous Steiner symmetrization (CStS); see \cite{Bro1995MR1330619, Broc2000MR1758811, Bro2016MR3509374, Sol2020MR4175494}. We will follow the presentation in \cite{Bro2016MR3509374}. Below is the form we will use; the original result applies to more general cases.

We start with some notation. Let $\mathcal{L}^N$ denote $N$-dimensional Lebesgue measure. By $\mathcal{M}(\R^2)$ we denote the family of Lebesgue measurable sets in $\R^2$ with finite measure. For a function $u:\R^2\to \R$, let $\{u>a\}$ and $\{b\ge u>a\}$ denote the sets $\left\{x\in \R^2: u(x)>a \right\}$ and $\left\{x\in \R^2: b\ge u(x)>a \right\}$, respectively, ($a, b\in \R$, $a<b$). Let $\mathcal{S}(\R^2)$ denote the class of real, nonnegative measurable functions $u$ satisfying
\begin{equation*}
  \LL^2(\{u>c\})<+\infty,\ \ \forall\, c>\inf u.
\end{equation*}

The following is the definition of classical Steiner symmetrization; see, for instance, \cite{Bro2007MR2569330, Kaw1985MR810619, Lie2001MR1817225}.
\begin{definition}[Steiner symmetrization]
\ \ \
  \begin{itemize}
    \item [(i)]For any set $M\in \M(\R)$, let
    \begin{equation*}
      M^*:=\left(-\frac{1}{2}\LL^1(M),\  \frac{1}{2}\LL^1(M)\right).
    \end{equation*}
    \item [(ii)]Let $M\in \M(\R^2)$. For every $x_2\in \R$, let
    \begin{equation*}
      M(x_2):=\left\{x_1\in \R: (x_1, x_2)\in M \right\}.
    \end{equation*}
    The set
    \begin{equation*}
      M^*:=\left\{ x=(x_1, x_2): x_1\in \left(M(x_2) \right)^*, x_2\in \R\right\}.
    \end{equation*}
    is called the Steiner symmetrization of $M$ (with respect to $x_1$).
    \item [(iii)]If $u\in \s$, then the function
    \begin{equation*}
      u^*(x):=\begin{cases}
                \sup \left\{c>\inf u: x\in \left\{u>c \right\}^*\right\}, & \mbox{if }\  x\in \bigcup_{c>\inf u} \left\{u>c \right\}^*, \\
                \inf u, & \mbox{if }\  x\not\in \bigcup_{c>\inf u} \left\{u>c \right\}^*,
              \end{cases}
    \end{equation*}
   is called the Steiner symmetrization of $u$ (with respect to $x_1$).
  \end{itemize}
\end{definition}

\begin{definition}[Continuous symmetrization of sets in $\M(\R)$]
  A family of set transformations
  \begin{equation*}
    \T_t:\  \M(\R)\to \M(\R),\ \ \  0\le t\le +\infty,
  \end{equation*}
is called a continuous symmetrization on $\R$ if it satisfies the following properties: ($M, E\in \M(\R)$, $0\le s, t\le +\infty$)
\begin{itemize}
  \item [(i)]Equimeasurability property:\, $\LL^1(\T_t(M))=\LL^1(M)$,

    \smallskip
  \item [(ii)]Monotonicity property:\, If $M\subset E$, then $\T_t(M)\subset \T_t(E)$,

    \smallskip
  \item [(iii)]Semigroup property:\, $\T_t(\T_s(M))=\T_{s+t}(M)$,

    \smallskip
  \item [(iv)]Interval property:\, If $M$ is an interval $[x-R,\ x+R]$, ($x\in \R$, $R>0$), then $\T_t(M):=[xe^{-t}-R,\ xe^{-t}+R]$,

    \smallskip
  \item [(v)]Open/compact set property: If $M$ is open/compact, then $\T_t(M)$ is open/compact.
\end{itemize}
\end{definition}

The existence and uniqueness of the family $\T_t$, $0\le t \le +\infty$, can be found in \cite [Theorem 2.1]{Broc2000MR1758811}.

\begin{definition}[Continuous Steiner symmetrization (CStS)]\label{csts}
  \ \ \
  \begin{itemize}
    \item [(i)]Let $M\in \M(\R^2)$. The family of sets
    \begin{equation*}
      \T_t(M):=\left\{x=(x_1, x_2): x_1\in \T_t(M(x_2)), x_2\in \R \right\},\ \ \ 0\le t\le +\infty,
    \end{equation*}
    is called the continuous Steiner symmetrization (CStS) of $M$ (with respect to $x_1$).
    \item [(ii)]Let $u\in \s$. The family of functions $\T_t(u)$, $0\le t \le +\infty$, defined by
    \begin{equation*}
      \T_t(u)(x):=\begin{cases}
                \sup \left\{c>\inf u: x\in \T_t\left(\left\{u>c \right\}\right)\right\}, & \mbox{if }\ x\in \bigcup_{c>\inf u} \T_t\left(\left\{u>c \right\}\right), \\
                \inf u, & \mbox{if }\ x\not\in \bigcup_{c>\inf u} \T_t\left(\left\{u>c \right\}\right),
              \end{cases}
    \end{equation*}
    is called CStS of $u$ (with respect to $x_1$).
  \end{itemize}
\end{definition}
For convenience, we will henceforth simply write $M^t$ and $u^t$ for the sets $\T_t(M)$, respectively for the functions $\T_t(u)$, ($t\in [0,+\infty]$). Note that if $u, v\in \s$ are two continuous functions with compact support satisfying $\text{supp}(u) \cap \text{supp}(v)=\varnothing$, then $(u+v)^t=u^t+v^t$ for all sufficiently small $t$.

Below we summarize basic properties of CStS, which have been proved by Brock in \cite{Bro1995MR1330619, Broc2000MR1758811}.

\begin{proposition}\label{pro0}
  Let $M\in \M(\R^2)$, $u,v\in \s,\, t\in [0,+\infty]$. Then
  \begin{itemize}
    \item [(1)]Equimeasurability:
    \begin{equation*}
      \LL^2(M)=\LL^2(M^t)\ \ \ \text{and}\ \ \ \left\{u^t>c \right\}=\left\{u>c \right\}^t,\  \forall\, c>\inf u.
    \end{equation*}

    \smallskip
    \item [(2)]Monotonicity: If $u\le v$, then $u^t\le v^t$.

    \smallskip
    \item [(3)]Commutativity: If $\phi: [0, +\infty)\to [0, +\infty)$ is bounded and nondecreasing with $\phi(0)=0$, then
    \begin{equation*}
      \phi(u^t)=[\phi(u)]^t.
    \end{equation*}

    \smallskip
    \item [(4)]Homotopy:
    \begin{equation*}
      M^0=M,\ \ \  u^0=u,\ \ \ M^\infty =M^*,\ \ \ u^\infty=u^*.
    \end{equation*}
Furthermore, from the construction of the CStS it follows that, if $M=M^*$ or $u=u^*$, then $M^t=M$, respectively, $u=u^t$ for all $t\in [0, +\infty]$.

    \smallskip
        \item [(5)]Cavalieri's pinciple: If $F$ is continuous and if $F(u)\in L^1(\R^2)$ then
        \begin{equation*}
          \int_{\R^2} F(u)\,\d x= \int_{\R^2} F(u^t)\,\d x.
        \end{equation*}

    \smallskip
    \item [(6)]Continuity in $L^p$: If $t_n\to t $ as $n\to +\infty$ and $u\in L^p(\R^2)$ for some $p\in [1, +\infty)$, then
    \begin{equation*}
      \lim_{n\to +\infty}\|u^{t_n}-u^t\|_p=0.
    \end{equation*}

    \smallskip
        \item [(7)]Nonexpansivity in $L^p$: If $u, v\in L^p(\R^2)$ for some $p\in [1, +\infty)$, then
        \begin{equation*}
          \|u^{t}-v^t\|_p\le \|u-v\|_p.
        \end{equation*}

          \smallskip
        \item [(8)]Hardy-Littlewood inequality: If $u, v\in L^2(\R^2)$ then
        \begin{equation*}
           \int_{\R^2} u^t v^t\,\d x\ge \int_{\R^2} u v\,\d x.
        \end{equation*}

         \smallskip
        \item [(9)]If $u$ is Lipschitz continuous with Lipschitz constant $L$, then $u^t$ is Lipschitz continuous, too, with Lipschitz constant less than or equal to $L$.

         \smallskip
        \item [(10)] If $\text{supp}\, (u)\subset B_R$ for some $R>0$, then we also have $\text{supp}\, (u^t)\subset B_R$. If, in addition, $u$ is Lipschitz continuous with Lipschitz constant $L$, then we have
            \begin{equation*}
              |u^t(x)-u(x)|\le LR\, t,\ \ \ \forall\, x\in B_R.
            \end{equation*}
           Furthermore, there holds
           \begin{equation*}
             \int_{B_R}G(|\nabla u^t|)\,\d x \le  \int_{B_R}G(|\nabla u|)\,\d x,
           \end{equation*}
           for every convex function $G: [0, +\infty) \to [0, +\infty)$ with $G(0)=0$.
  \end{itemize}
\end{proposition}

\subsection{Local symmmetry}
Following Brock \cite{Broc2000MR1758811}, we introduce a local version of symmetry for a function $u\in \s$.
\begin{definition}[Local symmetry in a certain direction]\label{def2-1}
Let $u\in \s$ and continuously differentiable on $\{x: 0<u(x)<\sup_{\R^2}u \}$, and suppose that this last set is open. Further, suppose that $u$ has the following property. If $y=(y_1, y_2)\in \R^2$ with
  \begin{equation*}
    0<u(y)<\sup_{\R^2} u,\ \ \ \partial_1 u(y)>0,
  \end{equation*}
  and $\tilde{ y}$ is the (unique) point satisfying
  \begin{equation*}
    \tilde{y}=(\tilde{y}_1, y_2),\ \ \ \tilde{y}_1>y_1,\ \ \ u(y)=u(\tilde{y})< u(s, y_2),\ \ \forall\, s\in (y_1,\tilde{y}_1),
  \end{equation*}
  then
  \begin{equation*}
  \partial_2 u(y)=\partial_2u(\tilde{y}),\ \ \ \partial_1u(y)=-\partial_1 u (\tilde{y}).
  \end{equation*}
  Then $u$ is called \emph{locally symmetric in the direction $x_1$}.
\end{definition}

Suppose that for arbitrary rotations $x\mapsto y=(y_1, y_2)$ of the coordinate system, $u$ is locally symmetric in the direction $y_1$. Then $u$ is said to be \emph{locally symmetric}. In other words, a function $u$ is said to be locally symmetric if it is locally symmetric in \emph{every} direction.

By \cite[Theorem 6.1]{Broc2000MR1758811}, locally symmetric functions exhibit a high degree of radial symmetry. Roughly speaking, it is radially symmetric and radially decreasing in some annuli (probably infinitely many) and flat elsewhere.
\begin{proposition}[ \cite{Broc2000MR1758811}, Theorem 6.1]\label{pro1}
  Let $u$ be locally symmetric, then we have the following decomposition:
    \begin{itemize}
    \item [(1)]$\displaystyle \{x: 0<u(x)<\sup_{\R^2}u \}=\bigcup_{k\in K} A_k \cup \{x\in V: \nabla u(x)=0\}$, \text{where}
    \begin{equation*}
      A_k=B_{R_k}(z_k)\backslash {Q_{r_k}(z_k)},\ \ \ z_k\in \R^2,\ \ \ 0\le r_k<R_k;
    \end{equation*}
     \item[(2)]$K$ is a countable set;
        \smallskip
    \item [(3)] the sets $A_k$ are pairwise disjoint;
    \smallskip
     \item [(4)]$u(x)=U_k(|x-z_k|)$, $x\in A_k$, where $U_k\in C^1([r_k, R_k])$;
         \smallskip
    \item [(5)]$U'_k(r)<0$ for $r\in (r_k, R_k)$;
        \smallskip
    \item [(6)]$u(x)\ge U_k(r_k),\ \forall\,x\in Q_{r_k}(z_k)$, $k\in K$.
  \end{itemize}
\end{proposition}

It can be seen that if $u\in C^1(\overline{B_R})$ ($u\not \equiv 0$) for some $R>0$ is locally symmetric when extended by zero outside $B_R$ and viewed as a function on $\mathbb{R}^2$, then the super-level sets $\{u>t\}$ $(t\ge 0)$ are countable unions of mutually disjoint discs, and $|\nabla u|=\text{const.}$ on the boundary of each of these discs. Below we show that a locally symmetric weakly superharmonic function must be radial.

\begin{lemma}\label{key0}
  Let $R > 0$, and let $u \in C^1(\overline{B_R})$ satisfy $u > 0$ in $B_R$ and $u = 0$ on $\partial B_R$. Suppose that $u$ is locally symmetric. Suppose also that $u$ is a weakly superharmonic function in the sense that
    \begin{equation*}
\int_{B_R} \nabla u \cdot \nabla \varphi \, \d x \geq 0
\end{equation*}
    for all non-negative functions $\varphi \in C_c^\infty(B_R)$. Then $u$ is a radially decreasing function, namely, $u=u(|x|)$ and
\begin{equation*}
  \frac{\d u}{\d r}<0\ \ \ \text{for}\ \, r=|x|\in (0, R).
\end{equation*}
    \end{lemma}

\begin{proof}
   Let $A_k=B_{R_k}(z_k)\backslash {Q_{r_k}(z_k)}\subset B_R$ denote an annulus in the decomposition specified in Proposition \ref{pro1}. It holds that $u(x)=U_k(|x-z_k|)$, $x\in A_k$. Note that $U_k'(r_k)=0$ and $U_k'(R_k)=0$ if $R_k<R$; see also \cite[Remark 2.2]{Bro2022MR4375744}. We claim that $R_k=R$, and consequently, $u$ is a radially decreasing function in $B_R$. If the assertion would not hold, then $U_k'(R_k)=0$. Notice that $u=c$ on $\partial B_{R_k}(z_k)$ for some constant $c\ge 0$. Since $u$ is weakly superharmonic, by the Hopf lemma, we have that $U_k'(R_k)<0$. This leads to a contradiction. The proof is thus complete.
   \end{proof}

\subsection{A symmetry criterion due to F. Brock}

The following symmetry criterion is attributed to Brock \cite{Broc2000MR1758811}.
\begin{proposition}[\cite{Broc2000MR1758811}, Theorem 6.2]\label{pro2}
  Let $R>0$. Let $u\in H^{1}_0(B_R) \cap C(\overline{B_R})$ be such that $u>0$ in $B_R$ and $u$ is continuously differentiable on $\left\{x\in B_R: 0<u(x)<\sup u\right\}$. Suppose that
  \begin{equation*}
    \lim_{t\to 0}\frac{1}{t}\left(\int_{B_R} |\nabla u|^2\,\d x-\int_{B_R} |\nabla u^t|^2\,\d x\right)=0.
  \end{equation*}
  Then $u$ is locally symmetric in the direction $x_1$, where $u$ is understood as a function on $\mathbb{R}^2$ extended by zero outside $B_R$.
\end{proposition}

\section{Radial symmetry of stationary vortex patches}\label{s3}
In this section, our objective is to establish the radial symmetry of stationary vortex patches. Given a stationary patch $D\subset \mathbb{D}$, let $u=\mathcal{G}[1_D]$ denote the associated stream function of the flow. Recalling \eqref{1-3}, we have $u=C_i$ on each connected component of
$\partial D$, where the constants can be different on different connected components.

\subsection{Warm-up: Radial symmetry of simply connected stationary patches}

We begin by examining the radial symmetry of simply connected stationary patches. Suppose that $D$ is a simply connected stationary vortex patch. Our goal is to prove that $D$ is a disc. This can be achieved using the classical symmetry result for positive solutions of semilinear equations; see \cite[Remark 2.2]{Wang2021MR4172850}. In fact, it is easy to see that
\begin{align}\label{3-1}
\begin{cases}
-\Delta u=1_D\ &  \text{in}\ \mathbb{D},\\
u=c\ &\text{on}\ \partial D,   \\
u=0\ &\text{on}\ \partial \mathbb{D},
\end{cases}
\end{align}
for some $c\in \R$. Applying the maximum principle, we conclude that $c > 0$ and $D=\{u>c\}$. Therefore, we have
\begin{align}\label{3-2}
\begin{cases}
-\Delta u=1_{\{u>c\}}\ &  \text{in}\ \mathbb{D},\\
u=0\ &\text{on}\ \partial \mathbb{D}
\end{cases}
\end{align}
for some $c>0$. By Corollary 3.9 in \cite{Fra2000MR1751289}, we observe that $u$ is radially decreasing, implying that $D$ is a disc centered at the origin.

\subsection{Radial symmetry of non-simply connected stationary patches}
In this subsection, we aim to prove the radial symmetry of a connected stationary patch $D$, where $D$ may be non-simply connected. Let $D\subset \mathbb{D}$ be a domain whose boundary consists of $n+1$ mutually disjoint Jordan curves. We denote the outer boundary of $D$ by $\Gamma_0$, and the inner boundaries by $\Gamma_i$ for $i = 1, \dots, n$. Set $V_i=\operatorname{int}(\Gamma_i)$ for $i=0, \dots, n$ (each $V_i$ is a simply connected domain). Note that $D$ has $n$ holes, namely, $V_i$ for $i=1, \dots, n$. Recall that $u=\mathcal{G}[1_D]$. By assumption on $D$, there exist constants $\{c_i\}_{i=0}^{n}$ such that
\begin{align}\label{3-1}
\begin{cases}
-\Delta u=1_D\ &  \text{in}\ \mathbb{D},\\
u=c_i\ &\text{on}\ \Gamma_i,\ \text{for}\ i=0,\dots, n,  \\
u=0\ &\text{on}\ \partial \mathbb{D}.
\end{cases}
\end{align}
Note that $u$ is a \emph{weakly superharmonic} function in $\D$, meaning that
\begin{equation*}
\int_\D \nabla u \cdot \nabla \varphi \, \d x \geq 0
\end{equation*}
for all non-negative functions $\varphi \in C_c^\infty(\D)$. By the maximum principle, we conclude that $c_i >c_0\ge0$ for $i = 1, \dots, n$. Furthermore, $u\equiv c_i$ in each $V_i$ for $i = 1, \dots, n$. Notice that we do not rule out the possibility that $\Gamma_0=\partial \D$. If $\Gamma_0 = \partial \D$, then $c_0 = 0$; otherwise, $c_0 > 0$.

\begin{proof}
Note that the desired radial symmetry will be achieved if we can show that $u$ is radial symmetric. In view of Lemma \ref{key0}, it is sufficient to prove that $u$ is locally symmetric, where $u$ is regarded as a function on $\mathbb{R}^2$, extended by zero outside $\mathbb{D}$. We will show that $u$ is locally symmetric in the direction $x_1$. The local symmetry in other directions can be established in a similar manner. Recalling Definition \ref{csts}, let $u^t$ denote the CStS of $u$ with respect to $x_1$. Thanks to Proposition \ref{pro2}, our task now is to prove that
\begin{equation}\label{3-0}
  \int_\D |\nabla u^t|^2\,\d x-\int_\D |\nabla u|^2\,\d x=o(t)
\end{equation}
as $t\to 0$. Here, the symbol $o(t)$ denotes any function such that $\lim_{t\to 0} o(t)/t=0$. By Proposition \ref{pro0} (10), we first have that
\begin{equation*}
  \int_\D |\nabla u^t|^2\,\d x-\int_\D |\nabla u|^2\,\d x\le 0.
\end{equation*}
So it suffices to show that
\begin{equation}\label{3-2}
  \int_\D |\nabla u^t|^2\,\d x-\int_\D |\nabla u|^2\,\d x\ge o(t)
\end{equation}
as $t\to 0$. Let $(u^t-u)$ be a test function. From \eqref{3-1}, we deduce that
\begin{equation}\label{3-3}
  \int_{\D}\nabla u \cdot \nabla(u^t-u)\,\d x=\int_{\D} (u^t-u)1_D\,\d x.
\end{equation}
Applying the Cauchy–Schwarz inequality, the left-hand side of \eqref{3-3} is estimated as follows
\begin{equation}\label{3-4}
   \int_{\D}\nabla u \cdot \nabla(u^t-u)\,\d x=\int_\D \nabla u\cdot \nabla u^t\,\d x-\int_\D|\nabla u|^2\,\d x\le \frac{1}{2}\int_\D|\nabla u^t|^2\,\d x-\frac{1}{2}\int_\D|\nabla u|^2\,\d x.
\end{equation}
Combining \eqref{3-2}, \eqref{3-3}, and \eqref{3-4}, we see that the problem reduces to showing that
\begin{equation*}
  \int_{\D} (u^t-u)1_D\,\d x\ge o(t)
\end{equation*}
as $t\to 0$. In fact, we will establish a stronger result, namely,
\begin{equation}\label{3-5}
  \int_{\D} (u^t-u)1_D\,\d x= o(t)
\end{equation}
as $t\to 0$. Notice that
\begin{equation}\label{3-6}
   \int_{\D} (u^t-u)1_D\,\d x   =\int_{D}(u^t-u)\,\d x=\int_{V_0}(u^t-u)\,\d x-\sum_{i=1}^{n}\int_{\overline{V_i}}(u^t-u)\,\d x.
\end{equation}
First, we show that
\begin{equation}\label{3-7}
  \int_{\overline{V_i}}(u^t-u)\,\d x=o(t),\ \ \ i=1,\dots, n,
\end{equation}
as $t\to 0$. Let $\phi_1(s)=1_{(c_i,\infty)}(s)$ and $\phi_2(s)=1_{[c_i,\infty)}(s)$ be two nondecreasing functions. In view of Proposition \ref{pro0} (3) and (6), we have $\phi_j(u^t)=[\phi_j(u)]^t$, and $\phi_j(u^t)\to \phi_j(u)$ in $L^1(\D)$ as $t\to 0$ for $j=1, 2$. It follows that $\LL^2(\overline{V_i}\cap\{u^t\not= c_i\})\to 0$ as $t\to 0$. By Proposition \ref{pro0} (10), we have $\|u^t-u\|_{L^\infty(\D)}\le Ct$ for some constant $C>0$ independent of $t$. Therefore, we conclude that
\begin{equation*}
 \left| \int_{\overline{V_i}}(u^t-u)\,\d x \right|\le \int_{\overline{V_i}}|u^t-u|\,\d x\le Ct \LL^2(\overline{V_i}\cap\{u^t\not= c_i\})=o(t),
\end{equation*}
from which \eqref{3-7} follows. Next we prove that
\begin{equation}\label{3-8}
\int_{V_0}(u^t-u)\,\d x=o(t)
\end{equation}
as $t\to 0$. By Proposition \ref{pro0} (3), we see that $(u^t-c_0)_+$ is a rearrangement of $(u-c_0)_+$. Hence
\begin{equation}\label{3-9}
  \int_{\D}(u^t-c_0)_+\d x=\int_{\D}(u-c_0)_+\d x,
\end{equation}
from which it follows that
\begin{equation*}
\begin{split}
   0 & = \int_{\D}(u^t-c_0)_+\d x-\int_{\D}(u-c_0)_+\d x \\
     &  =\int_{\D}\left(\int_{0}^{1}f(u(x)+\theta (u^t(x)-u(x)))\d \theta   \right)(u^t(x)-u(x))\,  \d x,
\end{split}
\end{equation*}
where $f(s):=1_{(c_0, \infty)}(s)$. Thus, we have
\begin{equation}\label{3-10}
\begin{split}
  & \int_{V_0}(u^t-u)\,\d x\\
   &=\int_{\D}f(u(x))(u^t(x)-u(x))\,\d x\\
  & =\int_{\D}f(u(x))(u^t(x)-u(x))\,\d x-\int_{\D}\left(\int_{0}^{1}f(u(x)+\theta (u^t(x)-u(x)))\d \theta   \right)(u^t(x)-u(x))\,  \d x \\
     & =\int_\D \left(f(u(x))- \int_{0}^{1}f(u(x)+\theta (u^t(x)-u(x)))\d \theta  \right)(u^t(x)-u(x))\,  \d x.
\end{split}
\end{equation}
Recalling that $\|u^t-u\|_{L^\infty(\D)}\le Ct$ and the monotonicity of $f$, it follows that
\begin{equation}\label{3-11}
\begin{split}
    &  \left|\int_\D \left(f(u(x))- \int_{0}^{1}f(u(x)+\theta (u^t(x)-u(x)))\d \theta  \right)(u^t(x)-u(x)) \, \d x \right|\\
     & \le Ct  \int_\D \left | f(u(x))- \int_{0}^{1}f(u(x)+\theta (u^t(x)-u(x)))\d \theta  \right|  \d x \\
     & \le Ct \int_\D \int_{0}^{1}\left |f(u(x)+\theta (u^t(x)-u(x)))-f(u(x))\right|\d \theta    \d x\\
     & \le Ct \int_{\D} |f(u(x))-f(u^t(x))|\,\d x.
\end{split}
\end{equation}
Note that $[f(u)]^t=f(u^t)$, which implies $\|f(u^t)-f(u)\|_{L^1(\D)}\to 0$ as $t\to 0$. Combining this with \eqref{3-10} and \eqref{3-11}, we obtain \eqref{3-8}. Thus, \eqref{3-5} follows directly from the combination of  \eqref{3-6}, \eqref{3-7} and \eqref{3-8}. Hereby, we have established \eqref{3-0}, and, by Proposition \ref{pro2}, $u$ is locally symmetric in the direction $x_1$. The same estimate \eqref{3-0} can be established for CStS in arbitrary directions. It follows that $u$ is locally symmetric. The proof is thus complete.
\end{proof}

For the convenience of later discussion, we refine and generalize the ideas demonstrated above. We have the following two useful results.
\begin{lemma}\label{key1}
  Let $u\in C^1(\overline{\D})$ be such that $u>0$ in $\D$ and $u=0$ on $\partial \D$. Recalling Definition \ref{csts}, let $u^t$ denote the CStS of $u$ with respect to $x_1$. Let $V$ be a measurable subset of $\D$. Suppose that $u\equiv c$ in $V$ for some constant $c>0$. Then
  \begin{equation*}
    \int_{V}|u^t-u|\,\d x=o(t)
  \end{equation*}
  as $t\to0$.
\end{lemma}

\begin{proof}
The desired result is obtained through a similar analysis to that used in the proof of \eqref{3-7}. Indeed, we first check that $\LL^2(V\cap\{u^t\not= c\})\to 0$ as $t\to 0$. In view of Proposition \ref{pro0} (10), we have
\begin{equation*}
  \int_{V}|u^t-u|\,\d x\le Ct \LL^2(V\cap\{u^t\not= c\})=o(t)
\end{equation*}
  as $t\to0$. The proof is thus complete.
\end{proof}

\begin{lemma}\label{key2}
  Let $u\in C^1(\overline{\D})$ be such that $u>0$ in $\D$ and $u=0$ on $\partial \D$. Let $\Gamma \subset \overline{\D}$ be a Jordan curve, and let $V=\operatorname{int}(\Gamma)$ denote the interior of $\Gamma$. Suppose that $u$ is weakly superharmonic in $\D$ and that $u=c$ on $\Gamma$ for some constant $c\ge0$. Recalling Definition \ref{csts}, let $u^t$ denote the CStS of $u$ with respect to $x_1$. Suppose also that $u>c$ in $V$, and that there exists a sequence of decreasing regular values ${\gamma_j}>c$, with $j = 1, 2, \dots$, such that $\lim_{j \to \infty} \gamma_j = c$. Then
  \begin{equation}\label{4-2}
    \int_{V}(u^t-u)\,\d x=o(t)
  \end{equation}
  as $t\to0$.
\end{lemma}

\begin{proof}
To prove \eqref{4-2}, it suffices to show that for any $\varepsilon>0$, we have
  \begin{equation}\label{4-3}
    \limsup_{t\to 0}\frac{\left|\int_{V}(u^t-u)\,\d x \right|}{t}\le \varepsilon.
  \end{equation}
  By Proposition \ref{pro0} (10), we have $\|u^t-u\|_{L^\infty(\D)}\le Ct$ for some constant $C>0$ independent of $t$. Since $\gamma_j$ is a regular value of $u$, the level set $u^{-1}(\{\gamma_j\})\cap V$ consists of a finite union of $C^1$ simple closed curves, denoted by $\{\ell_i\}_{i=1}^{n_j}$. Set $G_i=\operatorname{int}(\ell_i)$. Note that $u$ is weakly superharmonic in $\D$. It follows that
  \begin{equation*}
    V\backslash\left(\cup_{i=1}^{n_j}\overline{G_i}\right)=\{x\in V\mid c<u(x)<\gamma_j\}
  \end{equation*}
  which implies that $\LL^2\left(V\backslash\left(\cup_{i=1}^{n_j}\overline{G_i}\right)\right)\to 0$ as $j\to \infty$. Let $k$ be a sufficiently large integer such that
  \begin{equation}\label{4-3'}
    C\LL^2\left(V\backslash\left(\cup_{i=1}^{n_{k}}\overline{G_i}\right)\right)<\varepsilon.
  \end{equation}
   We compute
  \begin{equation}\label{4-4}
    \int_V (u^t-u)\,\d x=\sum_{i=1}^{n_{k}}\int_{G_i}(u^t-u)\,\d x+\int_{V\backslash\left(\cup_{i=1}^{n_k}\overline{G_i}\right)}(u^t-u)\,\d x.
  \end{equation}
Now let us prove that
\begin{equation}\label{4-5}
  \int_{G_i}(u^t-u)\,\d x=o(t),\ \ \ i=1,\dots, n_k,
\end{equation}
as $t\to0$. We can choose pairwise disjoint domains $U_i$, $i=1,\dots, n_k$, such that $\overline{G_i}\subset U_i\subset V$ and $\{u>\gamma_k\}\cap U_i=G_i$ for each $i=1,\dots, n_k$. By the definition of the CStS, we see that $(u^t-\gamma_k)_+1_{U_i}$ is a rearrangement of $(u-\gamma_k)_+1_{U_i}$ for sufficiently small $t$. Hence, for all sufficiently small $t$, it follows that
\begin{equation*}
 \int_{U_i}(u^t-\gamma_k)_+\,\d x=\int_{U_i}(u-\gamma_k)_+\,\d x,
\end{equation*}
which yields that
\begin{equation}\label{4-6}
\begin{split}
   &\int_{G_i}(u^t-\gamma_k)_+\,\d x-\int_{G_i}(u-\gamma_k)_+\,\d x \\
   & = \left[\int_{U_i}(u^t-\gamma_k)_+\,\d x- \int_{U_i}(u-\gamma_k)_+\,\d x\right]- \int_{U_i\backslash {G_i}}(u^t-\gamma_k)_+\,\d x \\
     & =-\int_{U_i\backslash {G_i}}\left[(u^t-\gamma_k)_+-(u-\gamma_k)_+ \right]\,\d x.
\end{split}
\end{equation}
Let $f(s):=1_{(\gamma_{k}, \infty)}(s)$. In view of Proposition \ref{pro0} (3) and (6), we have $f(u^t)=[f(u)]^t$, and $f(u^t)\to f(u)$ in $L^1(\D)$ as $t\to 0$. It follows that $\LL^2\left((U_i\backslash {G_i}\right)\cap \{u^t>\gamma_k\})\to 0$ as $t\to 0$. Hence, we have
\begin{equation}\label{4-7}
\begin{split}
      \left| \int_{U_i\backslash {G_i}}\left[(u^t-\gamma_k)_+-(u-\gamma_k)_+ \right]\,\d x\right|&\le \int_{U_i\backslash {G_i}}\left|(u^t-\gamma_k)_+-(u-\gamma_k)_+ \right|\,\d x \\
     & \le Ct \LL^2\left((U_i\backslash {G_i})\cap \{u^t>\gamma_k\}\right)=o(t).
\end{split}
\end{equation}
Combining \eqref{4-6} and \eqref{4-7}, we get
\begin{equation*}
\begin{split}
    o(t) &  =\int_{G_i}(u^t-\gamma_k)_+\,\d x-\int_{G_i}(u-\gamma_k)_+\,\d x \\
     & =\int_{G_i}\left(\int_{0}^{1}f(u(x)+\theta (u^t(x)-u(x)))\d \theta   \right)(u^t(x)-u(x))\,  \d x.
\end{split}
\end{equation*}
Therefore, we have
\begin{equation}\label{4-10}
\begin{split}
\int_{G_i}&(u^t-u)\,\d x=\int_{G_i}f(u(x))(u^t(x)-u(x))\,\d x\\
      =\int_{G_i}& \left(f(u(x))- \int_{0}^{1}f(u(x)+\theta (u^t(x)-u(x)))\d \theta  \right)(u^t(x)-u(x))\,  \d x+o(t).
\end{split}
\end{equation}
Recalling that $\|u^t-u\|_{L^\infty(\D)}\le Ct$ and the monotonicity of $f$, it follows that
\begin{equation}\label{4-11}
\begin{split}
    &  \left|\int_{G_i} \left(f(u(x))- \int_{0}^{1}f(u(x)+\theta (u^t(x)-u(x)))\d \theta  \right)(u^t(x)-u(x)) \, \d x \right|\\
     & \le Ct  \int_{G_i} \left | f(u(x))- \int_{0}^{1}f(u(x)+\theta (u^t(x)-u(x)))\d \theta  \right|  \d x \\
     & \le Ct \int_{G_i} |f(u(x))-f(u^t(x))|\,\d x\le Ct\|f(u^t)-f(u)\|_{L^1(\D)}=o(t).
\end{split}
\end{equation}
Thus, \eqref{4-5} follows directly from the combination of  \eqref{4-10} and \eqref{4-11}. In view of \eqref{4-4}, we conclude that
\begin{equation*}
   \int_V (u^t-u)\,\d x=o(t)+\int_{V\backslash\left(\cup_{i=1}^{n_k}\overline{G_i}\right)}(u^t-u)\,\d x,
\end{equation*}
which yields that
\begin{equation*}
    \limsup_{t\to 0}\frac{\left|\int_{V}(u^t-u)\,\d x \right|}{t} \le \limsup_{t\to 0}\frac{\left|\int_{V\backslash\left(\cup_{i=1}^{n_k}\overline{G_i}\right)}(u^t-u)\,\d x \right|}{t} \le  C\LL^2\left(V\backslash\left(\cup_{i=1}^{n_{k}}\overline{G_i}\right)\right)\le \varepsilon,
\end{equation*}
where \eqref{4-3'} was used in the final inequality. Thus, we have proved \eqref{4-3}, and the proof is complete.
\end{proof}
\begin{remark}
We note that, with suitable modifications to the proof of Lemmas \ref{key2} , it is likely possible to extend \eqref{4-2} to the form
\begin{equation*}
\int_{V}g(u)(u^t - u)\,\d x = o(t),
\end{equation*}
where $g=g_1+g_2$, with $g_1$ being continuous and $g_2$ a function of bounded variation. Such a generalization opens up the possibility of analyzing the radial symmetries of non-negative solutions to certain elliptic problems. Nevertheless, the form presented in Lemma \ref{key2} is sufficient for our purposes here.
\end{remark}

\subsection{Radial symmetry of disconnected stationary patches}
In this subsection, we generalize the previous result to a stationary patch $D$ with multiple disjoint components. Let $D\subset \mathbb{D}$ be an open set whose boundary consists of a finite collection of mutually disjoint Jordan curves. We denote by $\{D_i\}_{i=1}^m\, (m\ge 2)$  the connected components of $D$. Clearly, each $D_i$ is a domain whose boundary is composed of a finite collection of mutually disjoint Jordan curves. We denote the outer boundary of $D_i$ by $\Gamma_0^{(i)}$, and the inner boundaries by $\Gamma^{(i)}_j$ for $j = 1, \dots, n_i$. Set $V_j^{(i)}=\operatorname{int}(\Gamma^{(i)}_j)$ for $j=0, \dots, n_i$. Recall that $u=\mathcal{G}[1_D]$. By assumption on $D$, there exist constants $\{c^{(i)}_j: i=1, \dots, m;\ j=0, \dots, n_i\}$ such that
\begin{align}\label{4-1}
\begin{cases}
-\Delta u=1_D\ &  \text{in}\ \mathbb{D},\\
u=c^{(i)}_j\ &\text{on}\ \Gamma^{(i)}_j,\ \text{for}\ i=1, \dots, m;\ j=0, \dots, n_i,  \\
u=0\ &\text{on}\ \partial \mathbb{D}.
\end{cases}
\end{align}
Note that $u$ is a weakly superharmonic function in $\D$. By the maximum principle, we conclude that $c^{(i)}_j> 0$ for $i=1, \dots, m;\ j=1, \dots, n_i$. Moreover, we have $c_0^{(i)} \geq 0$ for $i = 1, \dots, m$, and at most one of them equals zero. For example, if $c_0^{(i)} = 0$, then $\Gamma_0^{(i)}$ must coincide with $\partial \D$.

We are now in a position to prove the radial symmetry of the stationary patch $D$. We shall proceed as in the proof of the previous subsection.
\begin{proof}
The desired radial symmetry will be obtained if we can show that $u$ is radial symmetric. By Lemma \ref{key0}, it is enough to prove that $u$ is locally symmetric. We will show that $u$ is locally symmetric in the direction $x_1$. The local symmetry in other directions can be established in a similar manner. Recalling Definition \ref{csts}, let $u^t$ denote the CStS of $u$ with respect to $x_1$. Thanks to Proposition \ref{pro2}, it suffices to prove that
\begin{equation}\label{4-13}
  \int_\D |\nabla u^t|^2\,\d x-\int_\D |\nabla u|^2\,\d x=o(t)
\end{equation}
as $t\to 0$. Note that
\begin{equation*}
  \int_\D |\nabla u^t|^2\,\d x-\int_\D |\nabla u|^2\,\d x\le 0.
\end{equation*}
So it suffices to show that
\begin{equation}\label{4-14}
  \int_\D |\nabla u^t|^2\,\d x-\int_\D |\nabla u|^2\,\d x\ge o(t)
\end{equation}
as $t\to 0$. Let $(u^t-u)$ be a test function. From \eqref{4-1}, we deduce that
\begin{equation}\label{4-15}
  \int_{\D}\nabla u \cdot \nabla(u^t-u)\,\d x=\int_{\D} (u^t-u)1_D\,\d x.
\end{equation}
Using the Cauchy–Schwarz inequality, we get
\begin{equation*}
   \int_{\D}\nabla u \cdot \nabla(u^t-u)\,\d x\le \frac{1}{2}\int_\D|\nabla u^t|^2\,\d x-\frac{1}{2}\int_\D|\nabla u|^2\,\d x.
\end{equation*}
In view of \eqref{4-15}, our task now is to show that
\begin{equation*}
  \int_{\D} (u^t-u)1_D\,\d x\ge o(t)
\end{equation*}
as $t\to 0$. We will establish a stronger result, namely,
\begin{equation*}
  \int_{\D} (u^t-u)1_D\,\d x= o(t)
\end{equation*}
as $t\to 0$. Note that
\begin{equation*}
\begin{split}
    \int_{\D} (u^t-u)1_D\,\d x &  =\sum_{i=1}^{m}\int_{D_i} (u^t-u)\,\d x \\
     & =\sum_{i=1}^{m}\left(\int_{V_0^{(i)}} (u^t-u)\,\d x-\sum_{j=1}^{n_i}\int_{\overline{V_j^{(i)}}} (u^t-u)\,\d x  \right).
\end{split}
\end{equation*}
Using Lemmas \ref{key1}, \ref{key2} and Sard's theorem, we deduce that
\begin{equation*}
\begin{split}
   \int_{V_0^{(i)}} (u^t-u)\,\d x & =o(t),\ \ \ i=1,\dots,m, \\
     \int_{\overline{V_j^{(i)}}} (u^t-u)\,\d x &  =o(t),\ \ \ i=1,\dots,m;\ j=1,\dots, n_i,
\end{split}
\end{equation*}
as $t\to 0$. Thus, we obtain
\begin{equation*}
  \int_{\D} (u^t-u)1_D\,\d x= o(t),
\end{equation*}
and the desired result follows. The proof is thus complete.
\end{proof}

\section{Radial symmetry of rotating vortex patches with $\Omega<0$}\label{s4}
In this section, we prove that a rotating vortex patch (with multiple disjoint patches) must be radially symmetric if the angular velocity $\Omega<0$. Let $D\subset \mathbb{D}$ be a rotating vortex patch with $\Omega<0$, where the boundary consists of a finite collection of mutually disjoint Jordan curves. Let $u=\mathcal{G}[1_D-2\Omega]$ denote the relative stream function of the flow. Recalling \eqref{1-4}, we have $u=C_i$ on each connected component of $\partial D$, where the constants may differ on different components (referred to as \emph{locally constant} on $\partial D$).

Note that $u$ satisfies the following problem:
\begin{align}\label{5-1}
\begin{cases}
-\Delta u=1_D-2\Omega\ &  \text{in}\ \mathbb{D},\\
u\ \text{is locally constant on}\  \partial D, &  \\
u=0\ &\text{on}\ \partial \mathbb{D}.
\end{cases}
\end{align}
With \eqref{5-1} in hand, and using a similar discussion as in the previous section, we can easily arrive at the desired radial symmetry of $D$. Indeed, the problem reduces to establishing the local symmetry of $u$.
By testing \eqref{5-1} with $(u^t-u)$ and applying Lemmas \ref{key1} and \ref{key2}, we conclude that $u$ is locally symmetric. The details of the proof are analogous to those in the previous section and are therefore omitted.

\section{Radial symmetry of rotating vortex patches with $\Omega\ge1/2$}\label{s5}

In this section, we show that a rotating vortex patch (with multiple disjoint patches) must be radially symmetric if the angular velocity $\Omega\ge 1/2$. Let $D\subset \mathbb{D}$ be a rotating vortex patch with $\Omega\ge 1/2$, where the boundary consists of a finite collection of mutually disjoint Jordan curves. Let $u=\mathcal{G}[1_D-2\Omega]$ denote the relative stream function of the flow. Set $\psi=-u$. Recalling \eqref{1-4}, we have
\begin{align*}
\begin{cases}
-\Delta \psi=2\Omega-1_D\ &  \text{in}\ \mathbb{D},\\
\psi\ \text{is locally constant on}\  \partial D, &  \\
\psi=0\ &\text{on}\ \partial \mathbb{D}.
\end{cases}
\end{align*}
Note that $\psi$ is weakly superharmonic in $\D$ since $\Omega\ge 1/2$. Following the argument for the case $\Omega\le 0$, with $u$ replaced by $\psi$, we can readily establish the desired radial symmetry of $D$. Since the proof is nearly identical, the details are omitted.

We conclude this section with a remark on the radial symmetry of multi-scale rotating vortex patches.
\begin{remark}\label{re5}
  Let $\omega(x)=\sum_{i=1}^{n}\alpha_i 1_{D_i}$, where $\alpha_i\in \R$ are constants, each $D_i\subset\D$ is a domain whose boundary consists of a finite collection of mutually disjoint Jordan curves, and $D_i\cap D_j=\varnothing$ if $i\not=j$. Set $\lambda=\min\{\alpha_1,\dots, \alpha_n\}$ and $\Lambda=\max\{\alpha_1,\dots, \alpha_n\}$.

  Assume that $\omega$ is a uniformly rotating solution of \eqref{1-1} with angular velocity $\Omega\in (-\infty, \lambda]\cup [\Lambda, \infty)$, in the sense that $\mathcal{G}[\omega]+\frac{\Omega}{2}|x|^2$ is locally constant on $\partial D_i$ for all $i=1, \dots, n$. Then $\omega$ is radially symmetric.

  The proof of this result is similar to that of the uniformly rotating patches case and is therefore omitted. It is worth noting that the two new thresholds for the angular velocity primarily ensure that the relative stream function $\mathcal{G}[\omega]+\frac{\Omega}{2}|x|^2$ is weakly superharmonic or subharmonic.
\end{remark}

\section{Radial symmetry of uniformly rotating smooth solutions}\label{s6}
In this section, we study the radial symmetry of uniformly rotating smooth solutions. Let $\omega_0\in C^2(\overline{\D})$. Assume $\omega(x, t)=\omega_0(e^{-i\Omega t} x)$ is a stationary/uniformly rotating smooth solution of \eqref{1-1}, in the sense that
\begin{equation*}
   \mathcal{G}[\omega_0]+\frac{\Omega}{2}|x|^2\equiv C_i\ \ \ \text{on each connected component of a regular level set of}\ \omega_0,
\end{equation*}
where $C_i$ may vary across different connected components of a given regular level set $\{\omega_0=c\}$. We will show that $\omega_0$ must be radially symmetric if $\Omega\le \inf \omega_0/2$ or $\Omega\ge \sup \omega_0/2$. For clarity, let us consider the two cases separately.

\subsection{The case $\Omega\le  \inf \omega_0 /2$} Note that if $\Omega\le \inf \omega_0 /2$, then the associated stream function $u= \mathcal{G}[\omega_0]+\frac{\Omega}{2}|x|^2$ is weakly superharmonic harmonic in $\D$. Note that the desired symmetry of $\omega_0$ will be achieved if we can show that $u$ is radially symmetric. By virtue of Lemma \ref{key0}, the problem reduces to showing that $u$ is locally symmetric. Recalling Definition \ref{csts}, let $u^t$ denote the CStS of $u$ with respect to $x_1$. Thanks to Proposition \ref{pro2}, our task now is to prove that
\begin{equation*}
  \int_\D |\nabla u^t|^2\,\d x-\int_\D |\nabla u|^2\,\d x=o(t)
\end{equation*}
By Proposition \ref{pro0} (10), we have that
\begin{equation*}
  \int_\D |\nabla u^t|^2\,\d x-\int_\D |\nabla u|^2\,\d x\le 0.
\end{equation*}
The problem reduces to showing that
\begin{equation}\label{6-2}
  \int_\D |\nabla u^t|^2\,\d x-\int_\D |\nabla u|^2\,\d x\ge o(t)
\end{equation}
Note that $u$ satisfies the following elliptic problem
\begin{align}\label{6-1}
\begin{cases}
-\Delta u=\omega_0-2\Omega\ &  \text{in}\ \mathbb{D},\\
u=0\ &\text{on}\ \partial \mathbb{D}.
\end{cases}
\end{align}
 Taking $(u^t - u)$ as a test function, we obtain
\begin{equation}\label{6-3}
  \int_\D |\nabla u^t|^2\,\d x-\int_\D |\nabla u|^2\,\d x=\int_\D \omega_0 (u^t-u)\,\d x.
\end{equation}
The Cauchy–Schwarz inequality yields that
\begin{equation}\label{6-4}
   \int_{\D}\nabla u \cdot \nabla(u^t-u)\,\d x\le \frac{1}{2}\int_\D|\nabla u^t|^2\,\d x-\frac{1}{2}\int_\D|\nabla u|^2\,\d x.
\end{equation}
So it remains to show that
\begin{equation*}
  \int_\D \omega_0 (u^t-u)\,\d x\ge o(t)
\end{equation*}
as $t\to 0$. We will establish a slightly stronger result, namely
\begin{equation}\label{6-5}
 \mathcal{I}(t):= \int_{\D} \omega_0(u^t-u)\,\d x= o(t)
\end{equation}
as $t\to 0$. To this aim, we employ an approximation technique. For any integer $k>1$, we can approximate $\omega_0$ by a step function $w_n$ of the form $w_k=\sum_{i=1}^{M_k}\alpha_i 1_{D_i}(x)$, which satisfies all the following properties (see, e.g., page 2997 in \cite{Gom2021MR4312192}):
\begin{itemize}
  \item [(a)]Each $D_i$ is a connected open domain with $C^2$ boundary and possibly has a finite number of holes, denoted by $\{V_j^{(i)}\}_{j=0}^{n_i}$;
  \item [(b)]Each connected component of $\partial D_i$ is a connected component of a regular level set of $ \omega_0$;
  \item [(c)]$D_i\cap D_j=\varnothing$ if $i\not=j$;
  \item [(d)]$\|w_k-\omega_0\|_{L^\infty(\D)}\le \frac{2}{k}\|\omega_0\|_{L^\infty(\D)}$.
\end{itemize}
We estimate $\mathcal{I}(t)$ as follows:
\begin{equation*}
\begin{split}
   \mathcal{I}(t) &  = \int_{\D} \omega_0(u^t-u)\,\d x \\
     & =  \int_{\D} \left[\omega_0-w_k\right](u^t-u)\,\d x+\int_{\D} w_k(u^t-u)\,\d x\\
     &:=\mathcal{I}_1(t)+\mathcal{I}_2(t).
\end{split}
\end{equation*}
By Proposition \ref{pro0} (10), we have $\|u^t-u\|_{L^\infty(\D)}\le Ct$ for some constant $C>0$ independent of $t$. For any $\varepsilon>0$, there exists a sufficiently large inter $k$, such that
\begin{equation}\label{6-7}
\begin{split}
    |\mathcal{I}_1(t)| &  =\left| \int_{\D} \left[\omega_0-w_k\right](u^t-u)\,\d x \right| \\
     & \le \pi \|w_k-\omega_0\|_{L^\infty(\D)} \|u^t-u\|_{L^\infty(\D)} \le \frac{2\pi}{n}\|\omega_0\|_{L^\infty(\D)} Ct<\varepsilon t
\end{split}
\end{equation}
for all $t\ge 0$. Let $k$ be fixed. For $\mathcal{I}_2(t)$, we have
\begin{equation*}
\begin{split}
  \int_{\D} w_k(u^t-u)1_D\,\d x & =\sum_{i=1}^{M_k} \alpha_i\int_{D_i} (u^t-u)\,\d x \\
     & =\sum_{i=1}^{M_k} \alpha_i\left(\int_{D_i} (u^t-u)\,\d x- \sum_{j=0}^{n_i}\int_{\overline{V_j^{(i)}}} (u^t-u)\,\d x \right).
\end{split}
\end{equation*}
Recall that $u$ is locally constant on each connected component of $\partial D_i$. By Lemmas \ref{key1} and \ref{key2}, we have
\begin{equation*}
  \int_{D_i} (u^t-u)\,\d x=o(t),\ \ \ \int_{\overline{V_j^{(i)}}} (u^t-u)\,\d x=o(t),\ \ \ i=1,\dots, M_n;\ j=0, \dots, n_i
\end{equation*}
as $t\to 0$. It follows that $\mathcal{I}_2(t)=o(t)$ as $t\to 0$. Combining this with \eqref{6-7}, we conclude that $\mathcal{I}(t) = o(t)$ as $t \to 0$, and thus \eqref{6-5} is established. Therefore, we have proved that $u$ is locally symmetric in the direction $x_1$. The local symmetry in other directions can be established in a similar manner. It follows that $u$ is locally symmetric. Since $u$ is weakly superharmonic in $\D$, it follows from Lemma \ref{key0} that $u$ is radially symmetric. Hence, $\omega_0 = -\Delta u$ must also be radially symmetric, and the proof is complete.

\subsection{The case $\Omega\ge \sup \omega_0/2$} Note that if $\Omega\ge \sup \omega_0 /2$, then the associated stream function $u= \mathcal{G}[\omega_0]+\frac{\Omega}{2}|x|^2$ is weakly subharmonic harmonic in $\D$. Hence $\psi=-u$ is weakly superharmonic harmonic in $\D$. By repeating the previous argument with $u$ replaced by $\psi$, we readily obtain the desired radial symmetry. Since the proof is nearly identical, the details are omitted.

\section*{Acknowledgements}

The authors express their gratitude to Yao Yao for her valuable and constructive feedback, which has greatly improved the quality of this manuscript.

\subsection*{Data Availability} Data sharing is not applicable to this article as no datasets were generated or analyzed during the current study.

\subsection*{Declarations}

\smallskip
\ \ \\

\noindent \textbf{Conflict of interest} The authors declare that they have no conflict of interest.

\bibliography{ref}

\end{document}